\documentclass[twosides, ammssymb, 10pt]{amsart}
\usepackage{amssymb}
\usepackage{color}
\usepackage{xcolor}
\usepackage{mathrsfs}
\usepackage{amsmath}
\usepackage{amsfonts}
\usepackage{amsfonts,amssymb}
\usepackage{amscd}
\usepackage{amsthm,amsmath}
\usepackage[dvips]{graphicx}
\usepackage{xypic}
\usepackage{latexsym,amssymb}  
\usepackage{graphicx}

\def\ch{\mathfrak h^*}

\def\lr#1{[\,#1\,]}
\def\dlr#1{\langle #1 \rangle}

\renewcommand{\mod}[1]{{\rm mod}\,#1}
\def\ch{{\rm char}}

\newtheorem{thm}{Theorem}[section]
\newtheorem{lem}[thm]{Lemma}
\newtheorem{prop}[thm]{Proposition}
\newtheorem{cor}[thm]{Corollary}
\newtheorem{defn}[thm]{Definition}
\newtheorem{exam}[thm]{Example}
\newtheorem{remark}[thm]{Remark}

\numberwithin{equation}{section}

\date{}
\begin{document}

\thispagestyle{empty}

\begin{center}
{\bf{ \LARGE  The Grothendieck ring of a class of 2$n^2$-dimension semisimple Hopf Algebras $H_{2n^2}$} \footnotetext { $\dag$
Corresponding author: slyang@bjut.edu.cn
%
}}

\bigbreak

\normalsize Jialei Chen$^{*}$,  Shilin Yang$^{*\dag},$ Dingguo Wang$^{\sharp}$

{\footnotesize\small\sl $^{*}$College of Applied Sciences,
Beijing University of Technology\\
\footnotesize\sl Beijing 100124, P. R.  China }

{\footnotesize\small\sl $^{\sharp}$School of Mathematical Sciences,
Qufu Normal University \\
\footnotesize\sl Qufu 273165, P. R.  China }
\end{center}
\begin{quote}
{\noindent\small{\bf Abstract.}
In this paper, we construct the Grothendieck ring of a class of 2$n^2$-dimension semisimple Hopf Algebras $H_{2n^2}$, which can be viewed as a generalization of the 8-dimension Kac-Paljutkin Hopf algebra $K_8$. All irreducible  $H_{2n^2}$-modules are classified. Furthermore, we
 describe the Grothendieck ring $r(H_{2n^2})$ by generators and relations explicitly.

}
\end{quote}

\section{Introduction}
In the 1960$^\prime$s, Kac and Paljutkin (see \cite{KP}) discovered a non-commutative and non-cocommutative semisimple Hopf algebra $K_8$ of dimension 8. Later, Masuoka (see \cite{MAS}) constructed this Hopf algebra as an extension of $k[C_2 \times C_2]$ by $k[C_2]$. Recently, using Ore extension(see \cite{BDG,Pan,Wl,WL1,WL2,WZZ,XWC}), an important method to constructing Hopf algebras, Pansera constructed an interesting class of semisimple Hopf algebras $H_{2n^2}$ in \cite{PAN}. These Hopf algebras $H_{2n^2}$ of dimension $2n^2$ are neither commutative nor cocommutative. In particular, the Hopf algebra $K_8$ is just the Hopf algebra $H_8$. Therefore, $H_{2n^2}$ can be viewed as a generalization of the Kac-Paljutkin 8-dimensional Hopf algebra $K_8$.

The representations of $H_8$ were studied in several papers such as \cite{AL} and \cite{SHI}. It turns out that the Hopf algebra $H_8$ has 4 one-dimensional representations and a single two-dimensional simple module. Recently, the weak Hopf algebra $\widetilde{H_8}$ corresponding to $H_8$ was constructed in \cite{SY}, the representation ring of $\widetilde{H_8}$ was described and the automorphism group of $r(\widetilde{H_8})$ was proved to isomorphic to the
dihedral group $D_6$ with order 12. 

As is well known, the tensor product of finite dimensional representations of a Hopf algebra plays an important role in the representation theory of Hopf algebras. Particularly, how to decompose a tensor product of two indecomposable modules into a direct sum of indecomposable representations has attracted numerous attentions. One method of addressing this problem is to take the tensor product as the multiplication of the Green ring (or the
representation ring) $r(H)$, and to study the ring properties of $r(H)$. In \cite{CIBILS}, Cibils classified the indecomposable modules over $k\mathbb{Z}_{n}(q)/I_{d}$, and gave the decomposition formulas of the tensor product of two indecomposable $k\mathbb{Z}_{n}(q)/I_{d}$-modules. Yang determined the representation type of a class of pointed Hopf algebras, classified all indecomposable modules of the simple pointed Hopf algebra $R(q,a)$, and gave the decomposition formulas of the tensor product of two indecomposable $R(q,a)$-modules(see \cite{YANG2}). It is noted that some results of $R(q,a)$ were recently extended to more general case of pointed Hopf algebras of rank one by Wang et al. (see \cite{ZLH}). Huang et al. computed the Clebsch-Gordan formulae and the Green rings of connected pointed tensor categories of finite type (see \cite{HOYZ}) and some tame hereditary pointed tensor categories which are not finite(see \cite{HY}). Li and Hu described the Green rings of the 2-rank Taft algebra(at $q=-1$)and its two relatives twisted by a representation theoretic analysis(see \cite{LH}). Chen, Van Oystaeyen and Zhang gave the Green rings of the Taft algebra $H_{n}(q)$ (see \cite{COZ}). Li and Zhang extended the results of \cite{COZ}, computed the Green rings of the Generalized Taft Hopf algebras $H_{n,d}$ by generators and generating relations, and determined all nilpotent elements in $r(H_{n,d})$ (see \cite{LZ}). Su and Yang (see \cite{SY1})
studied the Green rings of the weak Generalized Taft Hopf algebras $r(\mathfrak{w}^{s}(H_{n,d}))$, showing that the Green rings of the weak Generalized Taft Hopf algebras was much more complicated than its Grothendick ring.
Su and Yang (see \cite{SY2}) also characterized the representation ring of small quantum group $\bar{U}_q{(sl_2)}$
by generators and relations. It turns out that the representation ring of $\bar{U}_q{(sl_2)}$ is generated by infinitely many generators
subject to a family of generating relations. It is noted that most of the above results are obtained in the case of pointed Hopf algebras.

In this paper, we will study the Grothendieck ring of a class of 2$n^2$-dimension semisimple Hopf algebras $H_{2n^2}$, which is not pointed. All irreducible $H_{2n^2}$-modules are classified. Furthermore, we describe the Grothendieck ring $r(H_{2n^2})$ by generators and relations explicitly. It turns out that $r(H_{2n^2})$ is a commutative
ring generated by two elements subjecting to three relations for an odd number $n$, and three elements with five relations for an even number $n$.

The paper is organized as follows. In Section 1, we give the definition of $H_{2n^2}$  in \cite{PAN} by generators and relations. It is noted that $H_{2n^2}$ is a quasi-triangular Hopf algebra. A complete set of primitive central idempotents of $H_{2n^2}$ is constructed and its block decomposition is given. In Section 2, all the finite dimensional irreducible representations of $H_{2n^2}$ are classified and the decomposition formulas of the tensor product of two irreducible $H_{2n^2}$-modules are established. In Section 3, we describe the Grothendieck ring $r(H_{2n^2})$ by generators and relations explicitly. Finally we give some concrete examples for $n =2, 3, 4, \cdots, 8.$

Throughout this paper, we work over a fixed field $k$ containing an $n$-th primitive root $q$ of unity and $\ch k\nmid 2n^2$. For the theory of
Hopf algebras and quantum groups, we refer to \cite{KA,MA,MONT,SW}.

\section{The Hopf Algebras $H_{2n^2}$}
\label{sect-2}
In this section, let us recall the definition of the Hopf algebra $H_{2n^2}$ in \cite{PAN}.

Let $R$ be a Hopf algebra with the antipode $S$, $H = R[z; \sigma]$ be the Ore extension with the $\sigma-$derivation 0, where $\sigma$ is an automorphism of $R$ as an algebra. Suppose that
\begin{itemize}
  \item[(1)] $J\in R\otimes R$ such that $(\sigma, J)$ is a twisted homomorphism (the definition in detail see \cite{PAN}),
  \item[(2)] $\sigma \circ S = S \circ\sigma$ and $\sigma^2 = id$,
  \item[(3)] there exists $0 \neq t \in R$ such that (i) $\Delta(t)=J(\sigma\otimes \sigma)(J)(t \otimes t)$,  (ii) $t\sum\limits_{J}J^1S(J^2) = 1, $ (iii) $t\sum\limits_{J}\sigma(S(J^1)J^2)= 1,$
  where $J = \sum\limits_{J}J^1 \otimes J^2\in R\otimes R$.
\end{itemize}
Then $H/\langle z^2 - t\rangle$ is a finite dimensional Hopf algebra with the following structure
$$za=\sigma(a)z \hbox{ for all } a\in R, \ \Delta(z) = J(z \otimes z), \varepsilon(z)=1, \hbox{ and } S(z) = z.$$
%
%
%
%
In particular, if $R=k\dlr{x, y| x^n=y^n=1, xy=yx}$ is a group algebra, where $n > 1$. We take $q \in k$ to be a primitive $n$-th root of unity and $\sigma$ is an automorphism of $R$ as an algebra defined by $x^iy^s\rightarrow x^sy^i$, for $1 \leq i, s \leq n$.
We also take $$J:=\frac{1}{n}\sum_{i, j=0}^{n-1} q^{-ij} x^j\otimes y^i\in R\otimes R.$$ Then the pair $(\sigma, J)$ 
satisfies the above conditions (1)-(3). Therefore, we get a Hopf algebra $H_{2n^2}$ of dimension $2n^2$ as follows.

\begin{defn}{\rm (\cite{PAN})}\label{defn1-2} Let $n>1$ and $q$ be a primitive $n$-th root of unity. The Hopf algebra $H_{2n^2}$ is the associative algebra generated by $x, y$ and $z$, with the following relations
\begin{eqnarray*}
&&x^{n}=1, \ y^{n}=1,\\
&&xy=yx, \ zx=yz, \ zy=xz, \\
&&z^2=\frac{1}{n}\sum_{i,j=0}^{n-1}q^{-ij} x^iy^j.
\end{eqnarray*}
The co-multiplication, counit, and antipode are as follows:
\begin{eqnarray*}
  &&\Delta(x)=x\otimes x,\quad \epsilon(x)=1,\quad S(x)=x^{-1}\\
  &&\Delta(y)=y\otimes y,\quad \epsilon(y)=1,\quad S(y)=y^{-1}\\
  &&\Delta(z)=\frac{1}{n}\sum_{i,j=0}^{n-1}q^{-ij} x^iz\otimes y^jz, \quad\epsilon(z)=1,\quad S(z)=z.
\end{eqnarray*}
\end{defn}
One can check that
$$\left(\sum_{i=0}^{n-1} x^i\right)\left(\sum_{j=0}^{n-1} y^j\right)\left(1+z\right)$$
is the left and right integral of $H_{2n^2}$.
Therefore, it is easy to see that $H_{2n^2}$ is a non-commutative, non-cocommutative semisimple Hopf algebra with the
basis $$\{x^iy^j,\ x^iy^j z | 0\leq i, j\leq n-1\}.$$

For $0\leq j\leq n-1$, set
$$e_j=\frac{1}{n}\sum_{i=0}^{n-1}q^{-ij} x^i, \quad f_j=\frac{1}{n}\sum_{i=0}^{n-1}q^{-ij} y^i,$$
then $\{e_0,\ e_1, \cdots,e_{n-1}\}$ and $\{f_0,\ f_1, \cdots,f_{n-1}\}$ are orthogonal idempotents of $H_{2n^2}$ respectively.

Let $H$ be a finite dimensional Hopf algebra and $R\in H\otimes H$ an invertible element. The pair
$(H, R)$ is said to be a quasi-triangular Hopf algebra and $R$ is said to be a universal $R$-matrix of $H$, if the following three conditions are satisfied.
\begin{itemize}
  \item[(i)] $\Delta^{\prime}(h)=R\Delta(h) R^{-1},$ for all $h\in H$;
  \item[(ii)] $(\Delta\otimes id) (R)=R_{13}R_{23}$;
  \item[(iii)] $(id \otimes \Delta ) (R)=R_{13}R_{12}$;
\end{itemize}
Here $\Delta^{\prime}=T\circ\Delta, T: H\otimes H\to H\otimes H, T(a\otimes b)=b\otimes a$, and $R_{ij}\in H\otimes H\otimes H$ is given
by $R_{12}=R\otimes 1$, $R_{23}=1\otimes R$,  $R_{13}=(T\otimes id)(R_{23})$.

\begin{prop}\label{prop1-2} $H_{2n^2}$ is a quasi-triangular Hopf algebra.
\end{prop}
\begin{proof} Indeed, let
$$R=\sum_{i=0}^{n-1}e_i\otimes y^{-i}=\frac{1}{n}\sum_{i, j=0}^{n-1} q^{-ij} x^j\otimes y^{-i}.$$
It is easy to see that $J=R^{-1}$ and it is straightforward to check that
$R$ satisfies the above three conditions.
Therefore $H_{2n^2}$ is a quasi-triangular Hopf algebra.
\end{proof}
It is well known that $\{e_if_j, e_if_jz|0\leq i,j\leq n-1\}$ is also a basis of $H_{2n^2}$, and any element $a$ of $H_{2n^2}$ can be written as
$$a=\sum_{i,j=0}^{n-1}\big(a_{ij}e_if_j+b_{ij}e_if_jz\big).$$
Denote the center of $H_{2n^2}$ by $\mathbb{Z}(H_{2n^2})$. We have

\begin{lem}\label{lem2-3} An element
$$a=\sum_{i,j=0}^{n-1}\big(a_{ij}e_if_j+b_{ij}e_if_jz\big)\in \mathbb{Z}(H_{2n^2})$$
 if and only if $a_{ij}=a_{ji}$, and $b_{ij}=0$ for $i\neq j$.
\end{lem}
\begin{proof}
Assume that $a\in \mathbb{Z}(H_{2n^2})$, then $za=az$. Note that
\begin{eqnarray*}
za=z\sum_{i,j=0}^{n-1}\big(a_{ij}e_if_j+b_{ij}e_if_jz\big)
&=&\sum_{i,j=0}^{n-1}\big(a_{ij}e_jf_iz+b_{ij}e_jf_iz^2\big)\\
&=&\sum_{i,j=0}^{n-1}\big(a_{ij}e_jf_iz+b_{ij}e_jf_i\sum_{k=0}^{n-1}e_{k}y^{k}\big)\\
&=&\sum_{i,j=0}^{n-1}\big(a_{ij}e_jf_iz+b_{ij}q^{ij}e_jf_i\big).
\end{eqnarray*}
Similarly, we have
\begin{eqnarray*}
az
=\sum_{i,j=0}^{n-1}\big(a_{ij}e_if_jz+b_{ij}q^{ij}e_if_j\big).
\end{eqnarray*}
It follows that $a_{ij}=a_{ji}$ and $b_{ij}=b_{ji}$ for $i\neq j$.  On the one hand,
$$xa=x\sum_{i,j=0}^{n-1}\big(a_{ij}e_if_j+b_{ij}e_if_jz\big)
=\sum_{i,j=0}^{n-1}\big(a_{ij}q^{i}e_if_j+b_{ij}q^{i}e_if_jz\big),$$
$$ax=\sum_{i,j=0}^{n-1}\big(a_{ij}e_if_j+b_{ij}e_if_jz\big)x
=\sum_{i,j=0}^{n-1}\big(a_{ij}q^{i}e_if_j+b_{ij}q^{j}e_if_jz\big).$$
It follows that $b_{ij}=0$ for $i\neq j$ since $q^{i}\neq q^{j}$ when $i\neq j$.
On the other hand, if $b_{ij}=0$ if $i\ne j$, it is easy to see that $ay=ya$.

The proof is completed.
\end{proof}

\begin{prop}\label{prop2-4}  For the Hopf algebra $H_{2n^2}$, the set
$$\left\{\frac{1}{2}e_if_i\pm\frac{1}{2}q^{-\frac{i^2}{2}}e_if_iz, i=0,1,\cdots,n-1\right\}\bigcup \left\{e_if_j+e_jf_i, i<j\right\}$$
forms a complete set of primitive central idempotents.
\end{prop}
\begin{proof}
Let
$$c_0=\sum_{i<j}^{n-1}a_{ij}(e_if_j+e_jf_i), \quad c_1=\sum_{i=0}^{n-1}\big(a_{ii}e_if_i+b_{ii}e_if_iz\big).$$
By Lemma \ref{lem2-3}, an element $c\in \mathbb{Z}(H_{2n^2})$ if and only if $c=c_0+c_1.$
If $c$ is an idempotent in addition, i.e., $c^2=c$,  we have $c_0^2+c_1^2=c_0+c_1$. Note that $$c_0^2=\left(\sum_{i<j}^{n-1}a_{ij}(e_if_j+e_jf_i)\right)^2=\sum_{i< j}^{n-1}a_{ij}^2\left(e_if_j+e_jf_i\right),$$
 and
\begin{eqnarray*}
c_1^2&=&\sum_{i=0}^{n-1}\big(a_{ii}e_if_i+b_{ii}e_if_iz\big)^2\\
&=&\sum_{i=0}^{n-1}\big(a_{ii}^2e_if_i+2a_{ii}b_{ii}e_if_iz+b_{ii}^2e_if_iz^2\big)\\
&=&\sum_{i=0}^{n-1}\big(a_{ii}^2e_if_i+2a_{ii}b_{ii}e_if_iz+b_{ii}^2e_if_i(\sum_{j=0}^{n-1} x^jf_j)\big)\\
&=&\sum_{i=0}^{n-1}\big(a_{ii}^2e_if_i+2a_{ii}b_{ii}e_if_iz+q^{i^2}b_{ii}^2e_if_i\big).
\end{eqnarray*}
Hence $c^2=c$ if and only if $c_0^2=c_0$ and $c_1^2=c_1$.
One the other hand, $c_0^2=c_0$ implies that $a_{ij}=0$ or $a_{ij}=1$ for $i\neq j$.
Therefore
$$c_0=\sum_{\hbox{some pairs } (i, j) \hbox{ with } i<j} (e_if_j+e_jf_i).$$
 As for $c_1$,
the equality $c_1^2=c_1$ implies that $a_{ii}^2+q^{i^2}b_{ii}^2=a_{ii}$ and $2a_{ii}b_{ii}=b_{ii}$.

If $b_{ii}=0$, then $a_{ii}^2=a_{ii}$, and $a_{ii}=0$ or $a_{ii}=1$, it follows that $c_1=\sum\limits_{ \hbox{ some } 0\leq i\leq n-1} e_if_i;$

If $b_{ii}\neq0$, then $2a_{ii}=1$, and $b_{ii}=\pm\frac{1}{2}q^{-\frac{i^2}{2}}$, it follows that
$$c_1=\sum_{ \hbox{ some } 0\leq i\leq n-1}\left(\frac{1}{2}e_if_i\pm\frac{1}{2}q^{-\frac{i^2}{2}}e_if_iz\right ).$$

It is known that all the elements in Proposition \ref{prop2-4} are central idempotents and the sum of these elements is $1$. It is easy to see that
$H_{2n^2}(e_if_j)=(e_jf_i)H_{2n^2}$ is minimal as left or right ideal of $H_{2n^2}$ and
$\dim H_{2n^2}(e_if_j)=2$. It follows that
each central idempotents $e_if_j+e_jf_i$ generates a 4-dimensional ideal of $H_{2n^2}$.
 There are $\frac{n^2-n}{2}$ such central idempotents. Furthermore, there are $2n$ central idempotents $$\frac{1}{2}e_if_i\pm\frac{1}{2}q^{-\frac{i^2}{2}}e_if_i$$
generates one dimension ideal of $H_{2n^2}$.
The sum of the dimension of these ideals is
$$2n+4\cdot\frac{n^2-n}{2}=2n^2=\dim H_{2n^2}.$$
This implies central idempotents $e_if_j+e_jf_i$ and $\frac{1}{2}e_if_i\pm\frac{1}{2}q^{-\frac{i^2}{2}}e_if_i$
 are all primitive.

 The proof is completed.
\end{proof}
\begin{cor}\label{cor2-5} As an algebra, we have
$$H_{2n^2}=k^{\oplus 2n}\oplus M_2(k)^{\oplus \frac{n^2-n}{2}}.$$
\end{cor}

\section{representations of $H_{2n^2}$}
\label{sect-3}
As is known to all, $H_{2n^2}$ is semisimple.
In the section, we give all the finite dimensional irreducible
$H_{2n^2}$-modules and investigate the decomposition formulas of
the tensor product of two irreducible $H_{2n^2}$-modules.

Set
$$\sigma(m)=\left\{
             \begin{array}{ll}
               1, & \hbox{ if } 0\leq m\leq n-1; \\
               -1, & \hbox{if }  n\leq m\leq 2n-1.
             \end{array}
           \right.
$$
Let $S_m, m\in \mathbb{Z}_{2n}$ be a one-dimensional irreducible $H_{2n^2}$-module  with basis $v^{m}$, the actions of $H_{2n^2}$ on $S_m$ are

\begin{eqnarray*}
  &&x\cdot v^{m}=q^{m}v^{m},\\
  &&y\cdot v^{m}=q^{m}v^{m},\\
  &&z\cdot v^{m}=\sigma(m)q^{\frac{m^2}{2}}v^{m}.
\end{eqnarray*}
It is easy  to see that
$$e_j\cdot v^{m}=\frac{1}{n}\sum_{k=0}^{n-1}q^{-kj} x^k\cdot v^{m}\\
=\frac{1}{n}\sum_{k=0}^{n-1}q^{-kj+km} v^{m}\\
=\left\{
\begin{array}{ll}
0, \quad j\neq m(\mathrm{mod} ~n) ,   \\
v^{m}, j= m(\mathrm{mod} ~n).
\end{array}
\right.
$$
%

Let $S_{i,j}$ be the 2-dimensional irreducible $H_{2n^2}$-module
with the basis $v_1^{ij}$ and $v_2^{ij}$, where $0\leq i<j\leq n-1$, and
the actions of $H_{2n^2}$ on $S_{i,j}$ are
\begin{eqnarray*}
x(v_1^{ij},v_2^{ij})&=&(v_1^{ij},v_2^{ij})\left(
                                          \begin{array}{cc}
                                            q^i & 0 \\
                                            0 & q^j \\
                                          \end{array}
                                        \right),\\
y(v_1^{ij},v_2^{ij})&=&(v_1^{ij},v_2^{ij})\left(
                                          \begin{array}{cc}
                                            q^j & 0 \\
                                            0 & q^i \\
                                          \end{array}
                                        \right),\\
z(v_1^{ij},v_2^{ij})&=&(v_1^{ij},v_2^{ij})\left(
                                          \begin{array}{cc}
                                            0 & q^{ij} \\
                                            1 & 0 \\
                                          \end{array}
                                        \right).
\end{eqnarray*}
It is easy to see that
\begin{eqnarray*}
e_k\cdot v_1^{ij}&=&\left\{
\begin{array}{ll}
0, \quad k\neq i,  \\
v_1^{ij},\quad k=i.
\end{array}
\right.
\end{eqnarray*}
\begin{eqnarray*}
e_k\cdot v_2^{ij}&=&\left\{
\begin{array}{ll}
0, \quad k\neq j,   \\
v_2^{ij},\quad k=j.
\end{array}
\right.
\end{eqnarray*}

By Corollary \ref{cor2-5}, we have
\begin{prop}
The set
$$\left\{S_m,m\in \mathbb{Z}_{2n}, \} \cup \{S_{i,j},
0\leq i<j\leq n-1\right\}$$
 forms a complete list of non-isomorphic irreducible
$H_{2n^2}$-modules.
\end{prop}

Let $H$ be a finite dimensional Hopf algebra and
 $M$ and $N$  be two finite dimensional $H$-modules, then $M\otimes N$ is also an $H$-module defined by
$$h\cdot (m\otimes n)=\sum_{(h)} h_{(1)}\cdot m\otimes h_{(2)}\cdot n$$
for all $h\in H$ and $m\in M, n\in N$, where $\Delta(h)=\sum_{(h)}h_{(1)}\otimes h_{(2)}.$
By the Krull-Schmidt Theorem, any finite dimensional $H$-module can be decomposed into the direct sum of indecomposable $H$-modules.

Suppose that $S_m$ and $S_{m'}$ are two one dimensional irreducible $H_{2n^2}$-modules with the basis $v^m$ and $v^{m'}$ respectively. Then $S_m \otimes S_{m'}$ is also an $H_{2n^2}$-module with basis $v^m \otimes v^{m'}$, and the actions of $H_{2n^2}$
 on $S_m \otimes S_{m'}$ are as follows:
 \begin{eqnarray*}
  x\cdot(v^{m}\otimes v^{m'})&=&q^{m+m'}(v^{m}\otimes v^{m'}),\\
  y\cdot(v^{m}\otimes v^{m'})&=&q^{m+m'}(v^{m}\otimes v^{m'}),\\
  z\cdot(v^{m}\otimes v^{m'})&=&\big(\sum_{i=0}^{n-1}e_i\otimes y^i\big)(z\otimes z)(v^{m}\otimes v^{m'})
  =\sigma(m)\sigma(m')q^{\frac{m^2+{m'}^2}{2}}\big(\sum_{i=0}^{n-1}e_i\otimes y^i\big)(v^{m}\otimes v^{m'})\\
&=&\sigma(m)\sigma(m')q^{\frac{m^2+{m'}^2}{2}}\sum_{i=0}^{n-1}e_i\cdot v^{m}\otimes q^{im'}v^{m'}\\
&=&\left\{
\begin{array}{ll}
\sigma(m)\sigma(m') q^{\frac{m^2+{m'}^2}{2}}\cdot q^{mm'}(v^{m}\otimes v^{m'}), \quad 0\leq m \leq n-1 ,   \\
\sigma(m)\sigma(m') q^{\frac{m^2+{m'}^2}{2}}\cdot q^{(m-n)m'}(v^{m}\otimes v^{m'}), \quad n\leq m \leq 2n-1.
\end{array}
\right.
 \end{eqnarray*}
Therefore, we have
\begin{prop}\label{lem3-2} The following statements hold.
\begin{enumerate}
  \item $S_m\otimes S_{m'} \cong S_{(m+m')(\mod n)},$
  for all $0 \leq m,m'\leq n-1$;
  \item $S_m\otimes S_{m'} \cong S_{(m+m')},$
  for all $0 \leq m\leq n-1$, $n \leq m'\leq 2n-1$ and $n \leq m+m'\leq 2n-1$;
  or  $0 \leq m'\leq n-1$, $n \leq m\leq 2n-1$ and $n \leq m+m'\leq 2n-1$;
  \item $S_m\otimes S_{m'} \cong S_{(m+m'-n)},$ for all $0 \leq m\leq n-1$, $n \leq m'\leq 2n-1$ and $2n \leq m+m'\leq 3n-1$; or $0 \leq m'\leq n-1$, $n \leq m\leq 2n-1$ and $2n \leq m+m'\leq 3n-1$;
\item $S_m\otimes S_{m'} \cong S_{(m+m')((\mod 2n)},$ for all $n \leq m, m'\leq 2n-1$ and $2n \leq m+m'\leq 3n-1$;
\item $S_m\otimes S_{m'} \cong S_{(m+m')((\mod 3n)},$ for all $n \leq m, m'\leq 2n-1$ and $3n \leq m+m'\leq 4n-1$.
\end{enumerate}
\end{prop}

Now we deal with the tensor product of one-dimension irreducible $H_{2n^2}$-module and two-dimension irreducible $H_{2n^2}$-module. Suppose that $S_m$ and $S_{i,j}$ are two $H_{2n^2}$-modules, with basis $v^m$ and $v_1^{ij}, v_2^{ij}$ respectively. Then $S_m \otimes S_{i,j}$ is also an $H_{2n^2}$-module with basis $v^m \otimes v_1^{ij}$, and $v^m \otimes v_2^{ij}$. The actions of $H_{2n^2}$
 on $S_m \otimes S_{i,j}$ are as follows:
\begin{eqnarray*}
  x\cdot(v^{m}\otimes v_1^{ij})&=&q^{m+i}(v^{m}\otimes v_1^{ij}),\\
  y\cdot(v^{m}\otimes v_1^{ij})&=&q^{m+j}(v^{m}\otimes v_1^{ij}),\\
  z\cdot(v^{m}\otimes v_1^{ij})&=&\big(\sum_{k=0}^{n-1}e_k \otimes y^k\big)(z\otimes z)(v^{m}\otimes v_1^{ij})\\
&=&\sigma(m) q^{\frac{m^2}{2}}\big(\sum_{k=0}^{n-1}e_k v^{m}\otimes y^k v_2^{ij}\big)\\
&=&\sigma(m) q^{\frac{m^2}{2}+im} v^{m}\otimes v_2^{ij},
 \end{eqnarray*}
and
\begin{eqnarray*}
  x\cdot(v^{m}\otimes v_2^{ij})&=&q^{m+j}(v^{m}\otimes v_2^{ij}),\\
  y\cdot(v^{m}\otimes v_2^{ij})&=&q^{m+i}(v^{m}\otimes v_2^{ij}),\\
  z\cdot(v^{m}\otimes v_2^{ij})&=&\big(\sum_{k=0}^{n-1}e_k \otimes y^k\big)(z\otimes z)(v^{m}\otimes v_2^{ij})
  =\sigma(m) q^{\frac{m^2}{2}+ij}\big(\sum_{k=0}^{n-1}e_k v^{m}\otimes y^k v_1^{ij}\big)\\
&=&\sigma(m)q^{\frac{m^2}{2}+ij+jm}v^{m}\otimes v_1^{ij}
 \end{eqnarray*}
Let $\omega_1=v^{m}\otimes v_1^{ij}$, $\omega_2=\sigma(m) q^{\frac{m^2}{2}+im}v^{m}\otimes v_2^{ij}$, then
\begin{eqnarray*}
x(\omega_1,\omega_2)&=&(\omega_1,\omega_2)\left(
                                          \begin{array}{cc}
                                            q^{m+i} & 0 \\
                                            0 & q^{m+j} \\
                                          \end{array}
                                        \right),\\
y(\omega_1,\omega_2)&=&(\omega_1,\omega_2)\left(
                                          \begin{array}{cc}
                                            q^{m+j} & 0 \\
                                            0 & q^{m+i} \\
                                          \end{array}
                                        \right),\\
z(\omega_1,\omega_2)&=&(\omega_1,\omega_2)\left(
                                          \begin{array}{cc}
                                            0 & q^{(m+i)(m+j)} \\
                                            1 & 0 \\
                                          \end{array}
                                        \right).
\end{eqnarray*}

Similarly, $S_{i,j}\otimes S_m$ is a $H_{2n^2}$-module with basis $ v_1^{ij}\otimes v^m$, and $v_2^{ij}\otimes v^m$. The actions of $H_{2n^2}$
 on $S_{i,j} \otimes S_{m}$ are as follows:
\begin{eqnarray*}
  x\cdot(v_1^{ij}\otimes v^{m})&=&q^{i+m}(v_1^{ij}\otimes v^{m}),\\
  y\cdot(v_1^{ij}\otimes v^{m})&=&q^{j+m}(v_1^{ij}\otimes v^{m}),\\
  z\cdot(v_1^{ij}\otimes v^{m})&=&\sigma(m) q^{\frac{m^2}{2}+jm} v_2^{ij}\otimes v^{m},
 \end{eqnarray*}
and
\begin{eqnarray*}
  x\cdot(v_2^{ij}\otimes v^{m})&=&q^{j+m}(v_2^{ij}\otimes v^{m}),\\
  y\cdot(v_2^{ij}\otimes v^{m})&=&q^{i+m}(v_2^{ij}\otimes v^{m}),\\
  z\cdot(v_2^{ij}\otimes v^{m})&=&\sigma(m) q^{\frac{m^2}{2}+ij+im}v_1^{ij}\otimes v^{m}
 \end{eqnarray*}
Let $\omega_1=v_1^{ij}\otimes v^{m}$, $\omega_2=\sigma(m) q^{\frac{m^2}{2}+jm} v_2^{ij}\otimes v^{m}$, then
\begin{eqnarray*}
x(\omega_1,\omega_2)&=&(\omega_1,\omega_2)\left(
                                          \begin{array}{cc}
                                            q^{i+m} & 0 \\
                                            0 & q^{j+m} \\
                                          \end{array}
                                        \right),\\
y(\omega_1,\omega_2)&=&(\omega_1,\omega_2)\left(
                                          \begin{array}{cc}
                                            q^{j+m} & 0 \\
                                            0 & q^{i+m} \\
                                          \end{array}
                                        \right),\\
z(\omega_1,\omega_2)&=&(\omega_1,\omega_2)\left(
                                          \begin{array}{cc}
                                            0 & q^{(i+m)(j+m)} \\
                                            1 & 0 \\
                                          \end{array}
                                        \right).
\end{eqnarray*}
Therefore we have
\begin{prop}\label{lem3-3} For $H_{2n^2}$-modules $S_m$ and $S_{i,j}$, where $m\in \mathbb{Z}_{2n}$ and $0\leq i<j\leq n-1$, we have
$$S_m\otimes S_{i,j}\cong S_{m+i(\mod n),j+m(\mod n)}\cong S_{i,j}\otimes S_m$$
\end{prop}
Suppose that $S_{i,j}$ and $S_{k,l}$ are two $H_{2n^2}$-modules, with basis $v_1^{ij}, v_2^{ij}$ and $v_1^{kl}, v_2^{kl}$ respectively. Then $S_{i,j}\otimes S_{k,l}$ is a $H_{2n^2}$-module with basis $v_1^{ij} \otimes v_1^{kl}$, $v_2^{ij} \otimes v_2^{kl}$, $v_1^{ij} \otimes v_2^{kl}$ and $v_2^{ij} \otimes v_1^{kl}$. The actions of $H_{2n^2}$
 on $S_{i,j}\otimes S_{k,l}$ are as follows:
\begin{eqnarray*}
  x\cdot(v_1^{ij}\otimes v_1^{kl})&=&q^{i+k}(v_1^{ij}\otimes v_1^{kl}),\\
  y\cdot(v_1^{ij}\otimes v_1^{kl})&=&q^{j+l}(v_1^{ij}\otimes v_1^{kl}),\\
  z\cdot(v_1^{ij}\otimes v_1^{kl})&=&\sum_{s=0}^{n-1}e_s\cdot v_2^{ij}\otimes y^s\cdot v_2^{kl}\\
  &=&q^{jk} v_2^{ij}\otimes v_2^{kl},
 \end{eqnarray*}
and
\begin{eqnarray*}
  x\cdot(v_2^{ij}\otimes v_2^{kl})&=&q^{j+l}(v_2^{ij}\otimes v_2^{kl}),\\
  y\cdot(v_2^{ij}\otimes v_2^{kl})&=&q^{i+k}(v_2^{ij}\otimes v_2^{kl}),\\
  z\cdot(v_2^{ij}\otimes v_2^{kl})&=&q^{ij+kl}\sum_{s=0}^{n-1}e_s\cdot v_1^{ij}\otimes y^s\cdot v_1^{kl}\\
  &=&q^{ij+kl+il} v_1^{ij}\otimes v_1^{kl}.
 \end{eqnarray*}
Set $\omega_1=v_1^{ij}\otimes v_1^{kl}$, $\omega_2= q^{jk} v_2^{ij}\otimes v_2^{kl}$, then
\begin{eqnarray*}
x(\omega_1,\omega_2)&=&(\omega_1,\omega_2)\left(
                                          \begin{array}{cc}
                                            q^{i+k} & 0 \\
                                            0 & q^{j+l} \\
                                          \end{array}
                                        \right),\\
y(\omega_1,\omega_2)&=&(\omega_1,\omega_2)\left(
                                          \begin{array}{cc}
                                            q^{j+l} & 0 \\
                                            0 & q^{i+k} \\
                                          \end{array}
                                        \right),\\
z(\omega_1,\omega_2)&=&(\omega_1,\omega_2)\left(
                                          \begin{array}{cc}
                                            0 & q^{(i+k)(j+l)} \\
                                            1 & 0 \\
                                          \end{array}
                                        \right).
\end{eqnarray*}
Similarly, we have
\begin{eqnarray*}
  x\cdot(v_1^{ij}\otimes v_2^{kl})&=&q^{i+l}(v_1^{ij}\otimes v_2^{kl}),\\
  y\cdot(v_1^{ij}\otimes v_2^{kl})&=&q^{j+k}(v_1^{ij}\otimes v_2^{kl}),\\
  z\cdot(v_1^{ij}\otimes v_2^{kl})&=&q^{kl}\sum_{s=0}^{n-1}e_s\cdot v_2^{ij}\otimes y^s\cdot v_1^{kl}\\
  &=&q^{(j+k)l} v_2^{ij}\otimes v_1^{kl},
 \end{eqnarray*}
and
\begin{eqnarray*}
  x\cdot(v_2^{ij}\otimes v_1^{kl})&=&q^{j+k}(v_2^{ij}\otimes v_1^{kl}),\\
  y\cdot(v_2^{ij}\otimes v_1^{kl})&=&q^{i+l}(v_2^{ij}\otimes v_1^{kl}),\\
  z\cdot(v_2^{ij}\otimes v_1^{kl})&=&q^{ij}\sum_{s=0}^{n-1}e_s\cdot v_1^{ij}\otimes y^s\cdot v_2^{kl}\\
 &=&q^{i(j+k)} v_1^{ij}\otimes v_2^{kl}.
 \end{eqnarray*}
Set $\omega_3=v_1^{ij}\otimes v_2^{kl}$, $\omega_4= q^{(j+k)l} v_2^{ij}\otimes v_1^{kl}$, then
\begin{eqnarray*}
x(\omega_3,\omega_4)&=&(\omega_3,\omega_4)\left(
                                          \begin{array}{cc}
                                            q^{i+l} & 0 \\
                                            0 & q^{j+k} \\
                                          \end{array}
                                        \right),\\
y(\omega_3,\omega_4)&=&(\omega_3,\omega_4)\left(
                                          \begin{array}{cc}
                                            q^{j+k} & 0 \\
                                            0 & q^{i+l} \\
                                          \end{array}
                                        \right),\\
z(\omega_3,\omega_4)&=&(\omega_3,\omega_4)\left(
                                          \begin{array}{cc}
                                            0 & q^{(i+l)(j+k)} \\
                                            1 & 0 \\
                                          \end{array}
                                        \right).
\end{eqnarray*}
Therefore, we have
\begin{lem}\label{lem3-4} For two $H_{2n^2}$-modules $S_{i,j}$ and $S_{k,l}$, where $0\leq i<j,k<l\leq n-1$,  we have
$$S_{i,j}\otimes S_{k,l}\cong S_{i+k(\mod n),j+l(\mod n)}\oplus S_{i+l(\mod n),j+k(\mod n)}$$
provided that $i+k\neq j+l(\mod n)$ and $i+l\neq j+k(\mod n)$.
\end{lem}

Assume that $i+l\equiv j+k(\mod n)$, and $i+k\neq j+l(\mod n)$, set $m=i+l\equiv j+k(\mod n)$,
one sees that $0\leq m<n$.
Let $$\omega_1=v_1^{ij}\otimes v_2^{kl}+ q^{\frac{m^2}{2}-im} v_2^{ij}\otimes v_1^{kl}, \
\omega_2=v_1^{ij}\otimes v_2^{kl}-q^{\frac{m^2}{2}-im} v_2^{ij}\otimes v_1^{kl},$$
we have
\begin{eqnarray*}
  x\cdot\omega_1&=&q^{i+l}v_1^{ij}\otimes v_2^{kl}+ q^{\frac{m^2}{2}-im} \cdot q^{j+k} v_2^{ij}\otimes v_1^{kl}=q^m\omega_1,\\
   y\cdot\omega_1&=&q^{j+k}v_1^{ij}\otimes v_2^{kl}+ q^{\frac{m^2}{2}-im} \cdot q^{i+l}v_2^{ij}\otimes v_1^{kl}=q^m\omega_1,\\
   z\cdot\omega_1&=&q^{(j+k)l}v_2^{ij}\otimes v_1^{kl}+q^{(j+k)i} q^{\frac{m^2}{2}-im} v_1^{ij}\otimes v_2^{kl}\\
   &=&q^{m(m-i)}v_2^{ij}\otimes v_1^{kl}+ q^{\frac{m^2}{2}} v_1^{ij}\otimes v_2^{kl}=q^{\frac{m^2}{2}}\omega_1.
 \end{eqnarray*}
Similarly, we have
\begin{eqnarray*}
   x\cdot\omega_2&=&q^m\omega_2,\\
   y\cdot\omega_2&=&q^m\omega_2,\\
   z\cdot\omega_2&=&-q^{\frac{m^2}{2}}\omega_2.
 \end{eqnarray*}

Now it is easy to see that
\begin{lem}\label{lem3-5} For two $H_{2n^2}$-modules $S_{i,j}$ and $S_{k,l}$, where $0\leq i<j, k<l\leq n-1$,
we have
$$S_{i,j}\otimes S_{k,l}\cong S_{i+k(\mod n),j+l(\mod n)}\oplus S_{i+l(\mod n)}\oplus S_{j+k(\mod n)+n}$$
provided that $i+k\neq j+l(\mod n)$ and $i+l\equiv j+k(\mod n)$.
\end{lem}

Assume that $i+l\neq j+k(\mod n)$, and $i+k\equiv j+l(\mod n)$, denote
$m=i+k\equiv j+l(\mod n)$, one sees that  $0\leq m<n$. Let
$$\omega_1=q^{\frac{m^2}{2}-jk}v_1^{ij}\otimes v_1^{kl}+ v_2^{ij}\otimes v_2^{kl}, \omega_2=q^{\frac{m^2}{2}-jk}v_1^{ij}\otimes v_1^{kl}-v_2^{ij}\otimes v_2^{kl},$$
we have
\begin{eqnarray*}
  x\cdot\omega_1&=&q^{i+k}q^{\frac{m^2}{2}-jk}v_1^{ij}\otimes v_1^{kl}+ q^{j+l}v_2^{ij}\otimes v_2^{kl}=q^m\omega_1,\\
   y\cdot\omega_1&=&q^{j+l}q^{\frac{m^2}{2}-jk}v_1^{ij}\otimes v_1^{kl}+ q^{i+k}v_2^{ij}\otimes v_2^{kl}=q^m\omega_1,\\
   z\cdot\omega_1&=&q^{\frac{m^2}{2}}v_2^{ij}\otimes v_2^{kl}+q^{ij
   +kl+il} v_1^{ij}\otimes v_1^{kl}=q^{\frac{m^2}{2}}\omega_1.
 \end{eqnarray*}
Similarly, we have
\begin{eqnarray*}
   x\cdot\omega_2&=&q^m\omega_2,\\
   y\cdot\omega_2&=&q^m\omega_2,\\
   z\cdot\omega_2&=&-q^{\frac{m^2}{2}}\omega_2.
 \end{eqnarray*}

Now it is easy to see that
\begin{lem}\label{lem3-5} For two $H_{2n^2}$-modules $S_{i,j}$ and $S_{k,l}$, where $0\leq i<j, k<l\leq n-1$,
we have
$$S_{i,j}\otimes S_{k,l}\cong S_{i+k}\oplus S_{j+l}\oplus S_{i+l(\mod n),j+k(\mod n)}.$$
provided that $i+k\equiv j+l(\mod n)$ and $i+l\neq j+k(\mod n)$,
\end{lem}

Assume that $i+l\equiv j+k(\mod n)$, and $i+k\equiv j+l(\mod n)$.
Then $n$ has to be even and $j=i+\frac{n}{2}$, $l=k+\frac{n}{2}$.
So we have

\begin{lem}\label{lem3-6} For two $H_{2n^2}$-modules $S_{i,j}$ and $S_{k,l}$, where $0\leq i<j,k<l\leq n-1$,  we have
$$S_{i,j}\otimes S_{k,l}\cong S_{i+k}\oplus S_{j+l}\oplus S_{i+l(\mathrm{mod} ~n)}\oplus S_{j+k(\mod n)+n}$$
provided that $i+k\equiv j+l(\mod n)$ and $i+l\equiv j+k(\mod n).$
\end{lem}

In summary, we have the following  by Lemmas \ref{lem3-4}-\ref{lem3-6}.
\begin{prop}\label{prop3-6} For two $H_{2n^2}$-modules $S_{i,j}$ and $S_{k,l}$, we set
\begin{eqnarray*}
   I_1&=&\{0\leq i<j, k<l<n| i+k\equiv j+l(\mod n) \},\\
   I_2&=&\{0\leq i<j, k<l<n| i+l\equiv j+k(\mod n)\},
 \end{eqnarray*}
then we have
\begin{enumerate}
  \item $S_{i,j}\otimes S_{k,l}\cong S_{i+k(\mod n),j+l(\mod n)}\oplus S_{i+l(\mod n),j+k(\mod n)}$ if $i,j,k,l\notin I_1\cup I_2$;
  \item $S_{i,j}\otimes S_{k,l}\cong S_{i+k}\oplus S_{j+l}\oplus S_{i+l(\mod n),j+k(\mod n)}$ if
$i,j,k,l\in I_1-I_2$;
  \item $S_{i,j}\otimes S_{k,l}\cong S_{i+k(\mod n),j+l(\mod n)} \oplus S_{i+l(\mod n)}\oplus S_{j+k(\mod n)+n}$
 if $i,j,k,l\in I_2-I_1$;
 \item $n$ is even and $S_{i,j}\otimes S_{k,l}\cong S_{i+k}\oplus S_{j+l}\oplus S_{i+l(\mod n)}\oplus S_{j+k(\mod n)+n}$ if $i,j,k,l\in I_1\cap I_2$.
\end{enumerate}

\end{prop}

\section{Grothendieck ring of $H_{2n^2}$}
\label{sect-4}
Let $H$ be a finite dimensional Hopf algebra and $F(H)$ the free abelian group generated by the isomorphism classes $[M]$ of finite dimensional $H$-modules $M$.
The abelian group $F(H)$ becomes a ring if we endow $F(H)$ with a multiplication given by the tensor product $[M][N]=[M\otimes N]$. The Green ring(or representation ring) $r(H)$ of the Hopf algebra $H$ is defined to be the quotient ring of $F(H)$ modulo the relations $[M \oplus N ] = [M] + [N ]$. It follows that the Green ring $r(H)$ is an associative ring with identity given by
$[k_\varepsilon]$, the trivial 1-dimensional $H$-module. Note that $r(H)$ has a $\mathbb{Z}$-basis consisting of isomorphism classes of finite dimensional indecomposable $H$-modules.

The Grothendieck ring $G_{0}(H)$ of $H$ is the quotient ring of $F(H)$ modulo short exact sequences of $H$-modules, i.e., $[Y] =[X] +[Z]$ if $0 \rightarrow X\rightarrow Y\rightarrow Z\rightarrow 0$ is exact. The Grothendieck ring $G_{0}(H)$ possesses a basis given by isomorphism classes of simple $H$-modules. Particularly, if $H$ is a finite dimensional semi-simple Hopf algebra, then the Green ring $r(H)$
is equal to the Grothendieck ring $G_{0}(H)$ and is semi-simple(see \cite{LO}, \cite{WI}).

In this section we will describe the Grothendieck ring $r(H_{2n^2})$ of the Hopf algebra $H_{2n^2}$ explicitly by the generators and the generating relations.

By Proposition \ref{prop1-2}, we have $M\otimes N\cong N\otimes M$  for any finite dimensional $H_{2n^2}$-modules $M, N$. Therefore the Grothendieck ring $r(H_{2n^2})$ is commutative. Furthermore, $r(H_{2n^2})$ is
semisimple since $H_{2n^2}$ is a semisimple.

Let $F_{t}(y,z)$ be the generalized Fibonacci polynomials defined by
$$F_{t+2}(y,z)=zF_{t+1}(y,z)-yF_{t}(y,z)$$ for $t \geq 1$, while $F_{0}(y,z)=0, F_{1}(y,z)=1, F_{2}(y,z)=z$.  These generalized Fibonacci polynomials appeared in \cite{COZ} and \cite{LZ}.
\begin{lem}\label{lem4-2} \cite [Lemma 3.11]{COZ} For $t \geq 2$ have $$F_{t}(y,z)=\sum_{i=0}^{\left[\frac{t-1}{2}\right]}(-1)^{i} \left({t-1-i\atop i}\right)y^{i}z^{t-1-2i}$$
where $\left[\frac{t-1}{2}\right]$ denotes the biggest integer which is not bigger than $\frac{t-1}{2}$.
\end{lem}

 Let $a=[S_1]$, $b=[S_{n+1}]$ and $c=[S_{0,1}]$. By Lemma \ref{lem3-2} and Proposition \ref{prop3-6}, we have
\begin{lem} \label{lem4-3} The following statements hold in $r(H_{2n^2})$ if $n$ is odd.
\begin{enumerate}
  \item  For all $i\in \mathbb{Z}_{2n}$, we have
$$b^i=
  \left\{
    \begin{array}{ll}
     \left\{\begin{array}{ll}
                [S_{n+i}],\quad i<n;\\

               [S_{i}] , \quad i\geq n,
\end{array}
 \right., & i \hbox{ is odd;} \\
\\
     \left\{\begin{array}{ll}
                [S_{i}],\quad i<n;\\

                [S_{i-n}], \quad i>n; \\

               [S_{0}]=1 , \quad i=2n;
\end{array}
 \right., &  i \hbox{ is even.}
    \end{array}
  \right.$$

  \item $[S_{0,i+1}]$= $\left\{\begin{array}{ll}
                c^{2}+b^{n+1}+b,\quad i=1;\\

                c[S_{0,i}]-b[S_{0,i-1}], \quad 1< i\leq n-2;
\end{array}
 \right.$
 \item For all $1 \leq i < n$,
$$[S_{0,i}]b^j= \left\{\begin{array}{ll}
                [S_{j,i+j}],\quad i+j<n;\\

                [S_{0,j}], \quad i+j=n;
\end{array}
 \right.$$ and $[S_{0,i}]b^n=[S_{0,i}]$;
 \item $[S_{0,i}][S_{0,n-i}]=1+b^n+b^i[S_{0,n-2i}]$ for all $1\leq i\leq \frac{n-1}{2}$.
\end{enumerate}
\end{lem}
In particular, we have $cb^i=[S_{i,i+1}]$ for $0\leq i\leq n-2$, $cb^{n-1}=[S_{0,n-1}]$, and $cb^n=c.$

\begin{lem} \label{lem4-4} The following statements hold in $r(H_{2n^2})$ if $n$ is even.
\begin{enumerate}
  \item $a^i=[S_i]$, for all $1\leq i\leq n-1$, and $a^n=1$.
\item $b^i$= $\left\{\begin{array}{ll}
                [S_{n+i}],\quad 0<i<n, i=2k+1;\\

                [S_{i}], \quad 0<i<n, i=2k; \\

               1 , \quad i=n.
\end{array}
 \right.$
\item $a^{i}b=[S_{n+i+1}]$ for $0\leq i<n-1$ and $a^{n-1}b=[S_n]$.
  \item $cb^i=ca^i=[S_{i,i+1}]$ for $0\leq i\leq n-2$, $cb^{n-1}=ca^{n-1}=[S_{0,n-1}]$.
  \item $[S_{0,i+1}]$= $\left\{\begin{array}{ll}
                c^{2}-a-b,\quad i=1;\\

                c[S_{0,i}]-b[S_{0,i-1}], \quad 1< i\leq n-2;
\end{array}
 \right.$
 \item For all $1 \leq i < n$,
$[S_{0,i}]b^j$= $\left\{\begin{array}{ll}
                [S_{j,i+j}],\quad i+j<n;\\

                [S_{0,j}], \quad i+j=n;
\end{array}
 \right.$
 \item $[S_{0,i}][S_{0,n-i}]=1+a^{n-1}b+b^i[S_{0,n-2i}]$ for all $1\leq i\leq \frac{n-1}{2}$.
 \item $[S_{0,1}]^2=1+ab+a+b$ for $n=2$.
\end{enumerate}
\end{lem}
\begin{prop}\label{prop4-5} Suppose that $n$ is odd, we have
\begin{eqnarray*}
 [S_{0,m+2}]
&=&\sum_{i=0}^{\left[\frac{m+2}{2}\right]}(-1)^{i} \left({m+2-i\atop i}\right)b^{i}c^{m+2-2i}
-\sum_{i=0}^{\left[\frac{m}{2}\right]}(-1)^{i} \left({m-i\atop i}\right)b^{n+1+i}c^{m-2i}. \quad \quad \quad (3.4)
\end{eqnarray*}
for $0<m<n-2$.
\end{prop}
\begin{proof}
The result is proved by induction.
Note that
$$[S_{0,1}]=c, \ [S_{0,2}]=c^{2}-b^{n+1}-b,$$
 and
$$[S_{0,i+1}]=c[S_{0,i}]-b[S_{0,i-1}]$$
for $1<i<n-1$.
Thus
$$[S_{0,3}]=c[S_{0,2}]-b[S_{0,1}]=c^{3}-cb^{n+1}-bc-bc=c^{3}-3bc.$$
This equals to the right hand side of  $(3.4)$  for $m=1$. 
Hence $(3.4)$ holds for $m=0, 1$.

Now suppose that $(3.4)$ holds for $m, m+1$, then for $m+2$ we have
\begin{eqnarray*}
[S_{0,m+2}]&=&c[S_{0,m+1}]-b[S_{0,m}]\\
&=&c\bigg(F_{m+2}(b,c)-b^{n+1}F_{m}(b,c)\bigg)-b\bigg(F_{m+1}(b,c)-b^{n+1}F_{m-1}(b,c)\bigg)\\
&=&\bigg(cF_{m+2}(b,c)-bF_{m+1}(b,c)\bigg)-b^{n+1}\bigg(cF_{m}(b,c)-bF_{m-1}(b,c)\bigg)\\
&=&F_{m+3}(b,c)-b^{n+1}F_{m+1}(b,c)
\end{eqnarray*}
The proof is finished by Lemma \ref{lem4-2}.
\end{proof}
\begin{prop}\label{prop4-6} Suppose that $n$ is even, we have
\begin{eqnarray*}
 [S_{0,m+2}]
=\sum_{i=0}^{\left[\frac{m+2}{2}\right]}(-1)^{i} \left({m+2-i\atop i}\right)b^{i}c^{m+2-2i}
-a\sum_{i=0}^{\left[\frac{m}{2}\right]}(-1)^{i} \left({m-i\atop i}\right)b^{i}c^{m-2i}. \quad \quad (3.5)
\end{eqnarray*}
for $0<m<n-2$.
\end{prop}
\begin{proof}
It is noted that $[S_{0,2}]=c^{2}-a-b$ and $ca=cb$. Now the proof is similar to that of Proposition \ref{prop4-5}.
\end{proof}

As a consequence, we have
\begin{cor} \label{cor4-6} Keeping notations as above. Then
\begin{enumerate}
\item The set $\{b^{k}\mid 0 \leq k \leq 2n-1\}\cup \{c^{i}b^{j} \mid 1 \leq i\leq \frac{n-1}{2}, 0 \leq j \leq n-1\}$ forms a $\mathbb{Z}$-basis of $r(H_{2n^2})$ provided that $n$ is odd.
\item  The set $\{a^{i}b^j\mid 0 \leq i \leq n-1, j=0,1\}\cup \{c^{i}b^{j} \mid 1 \leq i< \frac{n}{2}, 0 \leq j \leq n-1\}\cup \{c^{\frac{n}{2}}b^{j} \mid 0 \leq j < \frac{n}{2}\}$ forms a $\mathbb{Z}$-basis of $r(H_{2n^2})$
    provided that $n$ is even.
\end{enumerate}
\end{cor}
\begin{proof}
(1) By Lemma \ref{lem4-3}, $b^{2n}=1$ and there is a one to one correspondence between the set $\{b^{i}\mid 0 \leq i \leq 2n-1\}$ and the set of one-dimensional irreducible $H_{2n^2}$ module $\{[S_{i}]\mid 0 \leq i \leq 2n-1\}$. Moreover, for all $0 <i < n$, $[S_{0,i}]b^{n-i}=[S_{0,n-i}]$, hence for $\frac{n-1}{2} <i < n$, $[S_{0,i}]$ can be obtained by $[S_{0,n-i}]b^{i}$. By Proposition \ref{prop4-5}, $[S_{0,i}]$ is a $\mathbb{Z}$-polynomial with $b$ and $c$, and the highest degree of $c$ in this polynomial is just $i$. Furthermore, $cb^n=c$ and $[S_{0,i}]b^{j}=[S_{j,i+j}],i +j<n$. Consequently, all the two-dimensional irreducible $H_{2n^2}$-modules $\{[S_{i,j}]\mid 0 \leq i <j\leq n-1\}$ can be obtained by a $\mathbb{Z}$-linear combination of
$$\left\{c^{i}b^{j} \mid 1 \leq i\leq \frac{n-1}{2}, 0 \leq j \leq n-1\right\}$$
and $\{b^i\mid 0\leq i\leq 2n-1\}$. The result is obtained.

(2) By Lemma \ref{lem4-4}, $a^{n}=1$ and $[S_i]=a^i$ when $0\leq i\leq n-1$, $[S_n]=a^{n-1}b$ and $[S_{n+i+1}]=a^{i}b$, $0\leq i< n-1$. Hence there is a one to one correspondence between the set $\{a^{i}b^j\mid 0 \leq i \leq n-1, j=0,1\}$ and the set of one-dimensional irreducible $H_{2n^2}$ module $\{[S_{i}]\mid 0 \leq i \leq 2n-1\}$. On the other hand, for all $0 <i < n$, $[S_{0,i}]b^{n-i}=[S_{0,n-i}]$, hence for $\frac{n}{2} <i < n$, $[S_{0,i}]$ can be obtained by $[S_{0,n-i}]b^{i}$. By Proposition \ref{prop4-6}, $[S_{0,i}]$ is a $\mathbb{Z}$-polynomial with $a, b$ and $c$, where $a$ appeared in $[S_{0,i}]$ if and only if $i$ is even, and the highest degree of $a$ in $[S_{0,i}]$ is 1, since $cb=ca$, while the highest degree of $c$ in $[S_{0,i}]$ is $i$. Hence $\{[S_{0,i}]\mid 1 \leq i \leq n-1\}$ is a $\mathbb{Z}$-linear combination of $\{c^{i}b^{j} \mid 1 \leq i\leq \frac{n}{2}, 0 \leq j \leq n-1\}$ and $\{ a^ib^{j}\mid i=0,1, 0\leq j\leq n-1\}$.
It is noted that
$$[S_{0,\frac{n}{2}}]b^{\frac{n}{2}}=[S_{0,\frac{n}{2}}], \ [S_{0,i}]b^{j}=[S_{j,i+j}],i +j<n,$$
and  the number of elements in
$$\left\{a^{i}b^j\mid 0 \leq i \leq n-1, j=0,1\}\cup \{c^{i}b^{j} \mid 1 \leq i< \frac{n}{2}, 0 \leq j \leq n-1\}\cup \{c^{\frac{n}{2}}b^{j} \mid 0 \leq j < \frac{n}{2}\right\} $$
is just $2n+\frac{n^2-n}{2}$. Hence we get the result.
\end{proof}
By Corollary \ref{cor4-6}, the ring $r(H_{2n^2})$ is the quotient ring of $\mathbb{Z}\lr{y, z}$ if $n$ is odd,  and $r(H_{2n^2})$ is the quotient of the ring $\mathbb{Z}\lr{x, y, z}$ if $n$ is even.

\begin{thm}\label{thm4-7} Suppose that $n$ is odd and $n\geq 3$, denote $m:=\frac{n-1}{2}$, then the Grothendieck ring $r(H_{2n^2})$ is isomorphic to the quotient ring of the ring $\mathbb{Z}\lr{y,z}$ module the ideal $I$ generated by the following elements
$$y^{2n}-1,\quad zy^{n}-z,$$ and
\begin{eqnarray*}
z^{m+1}-z^{m}y^{m+1}&+&\sum_{i=1}^{\left[\frac{m+1}{2}\right]}(-1)^{i} \left({m+1-i\atop i}\right)y^{i}z^{m+1-2i}\\
&-&\sum_{i=1}^{\left[\frac{m}{2}\right]}(-1)^{i} \left({m-i\atop i}\right)y^{m+1+i}z^{m-2i}\\
&-&y^{n+1}F_m(y,z)+y^{m+n+2}F_{m-1}(y,z).
\end{eqnarray*}
\end{thm}
\begin{proof}
By Corollary \ref{cor4-6}, when $n$ is odd, the ring $r(H_{2n^2})$ is generated by $b $ and $c$. Hence there is a unique ring epimorphism
$$\Phi: \mathbb{Z}\lr{ y, z}\rightarrow r(H_{2n^2})$$
from $\mathbb{Z}\lr{ y, z}$ to $r(H_{2n^2})$, such that
$$\Phi(y)=b=[S_{n+1}],\quad \Phi(z)=c=[S_{0,1}].$$
Since
$$b^{2n}=1,\quad cb^{n}=c,$$
By Lemma \ref{lem4-3}, we have
$$\Phi (y^{2n}-1)=0,\quad\Phi(zy^{n}-z)=0.$$ Note that by Lemma \ref{lem4-3} and proposition \ref{prop4-5}, $[S_{0,m+1}]=F_{m+2}(b,c)-b^{n+1}F_{m}(b,c)$,  and $[S_{0,m+1}]=[S_{0,m}]b^{m+1},$ thus we have
$$F_{m+2}(b,c)-b^{n+1}F_{m}(b,c)=b^{m+1}F_{m+1}(b,c)-b^{m+n+2}F_{m-1}(b,c),$$
i.e.,
\begin{eqnarray*}
c^{m+1}-c^{m}b^{m+1}&+&\sum_{i=1}^{\left[\frac{m+1}{2}\right]}(-1)^{i} \left({m+1-i\atop i}\right)b^{i}c^{m+1-2i}\\
&-&\sum_{i=1}^{\left[\frac{m}{2}\right]}(-1)^{i} \left({m-i\atop i}\right)b^{m+1+i}c^{m-2i}\\
&-&b^{n+1}F_m(b,c)+b^{m+n+2}F_{m-1}(b,c)=0.
\end{eqnarray*}
therefore $ \Phi$ maps the element
\begin{eqnarray*}
z^{m+1}-z^{m}y^{m+1}&+&\sum_{i=1}^{\left[\frac{m+1}{2}\right]}(-1)^{i} \left({m+1-i\atop i}\right)y^{i}z^{m+1-2i}\\
&-&\sum_{i=1}^{\left[\frac{m}{2}\right]}(-1)^{i} \left({m-i\atop i}\right)y^{m+1+i}z^{m-2i}\\
&-&y^{n+1}F_m(y,z)+y^{m+n+2}F_{m-1}(y,z)
\end{eqnarray*}
to 0.
It follows that $\Phi (I)=0,$ and $\Phi$ induces a ring epimorphism
$$\overline{\Phi}: \mathbb{Z}\lr{y,z}/I\rightarrow r(H_{2n^2}),$$
such that $\overline{\Phi}(\overline{v})=\Phi (v)$ for all $v\in \mathbb{Z}\lr{y, z}$, where $\overline{v}=\pi(v)$ and $\pi$ is the natural epimorphism $\mathbb{Z}\lr{y, z}\rightarrow \mathbb{Z}\lr{y, z}/I$.

Note that the ring $r(H_{2n^2})$ is the free $\mathbb{Z}$-module of rank $2n+\frac{n(n-1)}{2}$, with the $\mathbb{Z}$-basis $$\{c^{i}b^{j} \mid 1 \leq i\leq \frac{n-1}{2}, 1 \leq j \leq n-1\}\cup \{b^{k}\mid 0 \leq k \leq 2n-1\},$$
we can define a $\mathbb{Z}$-module homomorphism
$$\Psi: r(H_{2n^2})\rightarrow \mathbb{Z}\lr{y, z}/I,$$
$$c^{i}b^{j}\rightarrow\overline{z}^{i}\overline{y}^{j},~~ b^{k}\rightarrow\overline{y}^{k},$$
where $\quad 1 \leq i\leq \frac{n-1}{2}, 0 \leq j \leq n-1, 0 \leq k \leq 2n-1.$

On the other hand, as a free $\mathbb{Z}$-module, $\mathbb{Z}\lr{y,z}/I$ is generated by elements $\overline{z}^{i}\overline{y}^{j}$ and $\overline{y}^{k}, 1 \leq i\leq \frac{n-1}{2}, 0 \leq j \leq n-1, 0 \leq k \leq 2n-1,$ We have
$$\Psi\overline{\Phi}(\overline{z^{i}y^{j}})=\Psi\Phi(z^{i}y^{j})=\Psi(c^{i}b^{j})=
\overline{z}^{i}\overline{y}^{j}, $$
$$\Psi\overline{\Phi}(\overline{y^{k}})=\Psi\Phi(y^{k})=\Psi(b^{k})=
\overline{y}^{k},$$
for all $1 \leq i\leq \frac{n-1}{2}, 0 \leq j \leq n-1, 0 \leq k \leq 2n-1.$ Hence $\Psi\overline{\Phi}=id$, and $\overline{\Phi}$ is injective. Thus, $\overline{\Phi}$ is a ring isomorphism.
\end{proof}
Let $\dlr{f_1, \cdots, f_m}$ denote the ideal generated by polynomials $f_1, \cdots, f_m$ in some $\mathbb{Z}$-polynomial ring.
\begin{exam} The following three examples can be obtained easily from Theorem \ref{thm4-7}.
\begin{itemize}
  \item $r(H_{2\cdot 3^2})\cong \mathbb{Z}\lr{y, z}/\dlr{ y^6-1,zy^3-z,z^2-zy^2-y^4-y},$
  \item $r(H_{2\cdot 5^2})\cong \mathbb{Z}\lr{y, z}/\dlr{y^{10}-1,zy^5-z,z^3-z^2y^3-3zy+y^4+y^9},$
  \item $r(H_{2\cdot 7^2})\cong \mathbb{Z}\lr{y, z}/\dlr{ y^{14}-1,zy^7-z,z^4-z^3y^4+3zy^5-4z^2y+y^9+y^2}.$
\end{itemize}
\end{exam}
 Then we have the following
\begin{thm}\label{thm4-8} Suppose that $n$ is even.
\begin{itemize}
  \item[(a)] if $n=2$, then
  $$r(H_{8})\cong \mathbb{Z}\lr{x,y,z}/\dlr{y^2-1, x^2-y^{2}, zx-zy, z-zy,  z^2-x-y-xy-1};$$
  \item[(b)] if $n>2$ and we denote $m:=\frac{n}{2}$, then
  $r(H_{2n^2})$ is isomorphic to the quotient ring of the ring $\mathbb{Z}\lr{x,y,z}$ module the ideal $I$ generated by the following elements
$$x^{n}-1,\quad x^2-y^{2},\quad zx-zy$$ and
\begin{eqnarray*}
z^{m}-z^{m}y^{m}&+&\sum_{i=1}^{\left[\frac{m}{2}\right]}(-1)^{i} \left({m-i\atop i}\right)y^{i}z^{m-2i}\\
&-&\sum_{i=1}^{\left[\frac{m}{2}\right]}(-1)^{i} \left({m-i\atop i}\right)y^{m+i}z^{m-2i}\\
&-&xF_{m-1}(y,z)+xy^{m}F_{m-1}(y,z),
\end{eqnarray*}
\begin{eqnarray*}
z^{m+1}-z^{m-1}y^{m+1}&+&\sum_{i=1}^{\left[\frac{m+1}{2}\right]}(-1)^{i} \left({m+1-i\atop i}\right)y^{i}z^{m+1-2i}\\
&-&\sum_{i=1}^{\left[\frac{m-1}{2}\right]}(-1)^{i} \left({m-1-i\atop i}\right)y^{m+1+i}z^{m-1-2i}\\
&-&xF_{m}(y,z)+xy^{m+1}F_{m-2}(y,z).
\end{eqnarray*}
\end{itemize}
\end{thm}
\begin{proof}
Let
$$\Phi: \mathbb{Z}\lr{x, y, z}\rightarrow r(H_{2n^2})$$
be the ring epimorphism from $\mathbb{Z}\lr{x, y, z}$ to $r(H_{2n^2})$ such that
$$\Phi(x)=a=[S_{1}],\quad \Phi(y)=b=[S_{n+1}], \quad\Phi(z)=c=[S_{0,1}].$$
By Lemma \ref{lem4-4}, we have
$$a^{n}=1=b^n,\quad b^{2}=a^2, \quad ca=cb.$$
and when $n=2$, we have $c=cb$, and $c^2=1+ab+a+b$; when $n>2$, we have $[S_{0,m+1}]=[S_{0,m-1}]b^{m+1},$ and $[S_{0,m}]=[S_{0,m}]b^{m}$. By Proposition \ref{prop4-6} and Corollary \ref{cor4-6}, the result can be shown
as that of Theorem \ref{thm4-7}.
\end{proof}

\begin{exam}
We have the following examples.
\begin{itemize}
               \item $r(H_{2\cdot 4^2})\cong \mathbb{Z}\lr{x,y, z}/\dlr{y^4-1,x^2-y^2,zx-zy, z^2-z^2y^2-y+y^3-x+xy^2, z^3-zy^3-3yz};$
               \item $r(H_{2\cdot 6^2})\cong \mathbb{Z}\lr{x,y,z}/\dlr{y^{6}-1,x^2-y^2,zx-zy, z^3-z^3y^3-3yz+3y^4z, z^4-z^2y^4-4yz^2+y^2+y^5+xy+xy^4};$
               \item $r(H_{2\cdot 8^2})\cong \mathbb{Z}\lr{x, y, z}/\dlr{y^{8}-1,x^2-y^2,zx-zy, z^4-z^4y^4-4z^2y+4z^2y^5-xy^5-y^6+xy+y^2,
                   z^5-z^3y^5-5z^3y+5zy^2+3zy^6}.$
\end{itemize}
\end{exam}
\begin{remark}
The ring $r(H_{8})$ was considered earlier in \cite{SY}. It is the same as Theorem \ref{thm4-8}(a).
\end{remark}

\section{Conclusion}
We have described the Grothendieck ring of a class of 2$n^2$-dimension semisimple Hopf Algebras $H_{2n^2}$ by generators and relations explicitly. For the Grothendieck ring, it is interesting to determine its automorphism group and its Casimir number. This will be studied in our future works.

\section*{Acknowledgements}
The work is supported by National Natural
Science Foundation of China (Grant Nos. 11701019, 11671024 and 11471186 ) and
the Beijing Natural Science Foundation (Grant No. 1162002)


\begin{thebibliography}{111}

\bibitem{AL}\label{AL} Alaoui A E. The character table for a Hopf algebra arising from the Drinfel¡¯d double. J. Algebra, 2003, 265: 478-495.

\bibitem{BDG}\label{BDG} Beattie M, D$\breve{a}$sc$\breve{a}$lescu S, Gr$\ddot{u}$nenfelder L. Constructing Pointed Hopf Algebras by Ore Extensions. J. Algebra, 2000, 225: 743-770.

 \bibitem{COZ}\label{COZ}
Chen H, Oystaeyen F V, Zhang Y. The Green rings of Taft algebras. Proc. Amer. Math. Soc, 2014, 142(3): 765-775.

\bibitem{CIBILS}\label{CIBILS} Cibils C. A quiver quantum groups. Com. Math. Phys, 1993, 157(3): 459-477.

\bibitem{HOYZ}\label{HOYZ} Huang H, Oystaeyen F V, Yang Y et sl. The Green rings of pointed tensor categories of finite type. J. Pure Appl. Algebra, 2014, 218(2): 333-342.

\bibitem{HY}\label{HY} Huang H, Yang Y. The Green rings of minimal Hopf quivers. P. Edinburgh Math. Soc, 2014, 59(1):107-141

\bibitem{KP}\label{KP} Kac G I, Paljutkin V G. Finite ring groups. Trudy Moskov. Mat. Obshch, 1966, 15: 224-261.

\bibitem{KA}\label{KA} Kassel C.  Quantum groups. Springer-Verlag, New York, 1995.

\bibitem{LZ}\label{LZ} Li L, Zhang Y. The Green rings of the Generalized Taft algebras. Con. Math, 2013, 585(10): 275-288.

\bibitem{LH}\label{LH} Li Y, Hu N. The Green rings of the 2-rank Taft algebra and its two relatives twisted.
J. Algebra, 2014, 410: 1-35.

\bibitem{LO}\label{LO} Lorenz M. Representations of finite-dimensional Hopf algebras. J. Algebra, 1997, 188:
476-505.

\bibitem{WL2}\label{WL2} Lu D, Wang D. Ore extensions of quasitriangular Hopf group coalgebras. J. Algebra Appl, 2014, 13(6), 1450016.

\bibitem{MA}\label{MA}
Majid S. Foundations of quantum group theory. Cambridge Univ. Press, Cambridge, 1995.


\bibitem{MAS}\label{MAS} Masuoka A. Semisimple Hopf algebras of dimension 6, 8. Israel J. Math, 1995, 92(1-3): 361-373.

\bibitem{MONT}\label{MONT} Montgomery S. Hopf Algebras and their actions on rings. CBMS series in Math. 82, Amer. Math. Soc., Providence, RI, 1993.

\bibitem{Pan}\label{Pan} Panov A N. Ore extensions of Hopf algebras. Math. Notes, 2003, 74(3): 401-410.

\bibitem{PAN}\label{PAN}
Pansera D. A class of semisimple Hopf algebras acting on quantum polynomial algebras. arXiv: 1710.02729v1.

\bibitem{SHI}\label{SHI} Shi Y.
Finite dimensional Hopf algebras over Kac-Paljutkin algebra $H_8$. arXiv: 1612.03262v4.

\bibitem{SY}\label{SY}
Su D, Yang S. Automorphism group of representation ring of the weak Hopf algebra $\widetilde{H_8}$. Czech. Math. J. DOI: 10.21136/CMJ.2018.0131-17.

\bibitem{SY1}\label{SY1}
Su D, Yang S. Green rings of weak Hopf algebras based on generalized Taft algebras. Period. Math. Hung. DOI: 10.1007/s100998-017-221-0.

\bibitem{SY2}\label{SY2} Su D, Yang S. Representation ring of small quantum group $\bar{U}_q{(sl_2)}$. J. Math. Phys, 2017, 58, 091704.

\bibitem{SW}\label{SW} Sweedler M E. Hopf Algebras. Benjamin, New York, 1969.

\bibitem{WL1}\label{WL1} Wang D, Lu D. Ore extensions of Hopf group coalgebras. J. Korean Math. Soc,
2014, 51: 325-344.


\bibitem{WZZ}\label{WZZ} Wang D, Zhang J, Zhuang G. Primitive Cohomology of Hopf algebras. J. Algebra, 2016, 464: 36-96.

\bibitem{Wl}\label{Wl} Wang Z, Li L. Ore extensions of quasitriangular Hopf algebras. Acta Math. Sci,
2009, 29A(6): 1572-1579.

\bibitem{ZLH}\label{ZLH} Wang Z, You L, Chen H. Representations of Hopf-Ore extensions of group algebras and pointed Hopf algebras of rank one. Algebra Represent. Theory, 2015, 18(3): 801-830.


\bibitem{WI}\label{WI} Witherspoon S J. The representation ring of the quantum double of a finite group.
J. Algebra, 1996, 179: 305-329.

\bibitem{XWC}\label{XWC} Xu Y, Wang D, Chen J. Analogues of quantum Schubert cell algebras in PBW-deformations of quantum groups. J. Algebra Appl, 2016, 15(10), 1650179.

\bibitem{YANG2}\label{YANG2} Yang S. Representation of simple pointed Hopf algebras. J. Algebra Appl, 2004, 3(1):    91-104.
\end{thebibliography}
\end{document}